\documentclass[letterpaper, 10 pt, conference]{ieeeconf}

\IEEEoverridecommandlockouts	
\usepackage{cite}
\usepackage{amsmath,amssymb,amsfonts}
\usepackage{algorithmic}
\usepackage{balance}
\usepackage{hyperref}
\usepackage{epstopdf}
\usepackage{textcomp}
\usepackage{graphicx,color}
\usepackage{mathrsfs}
\usepackage[vlined,ruled]{algorithm2e}
\usepackage{subfigure}
\usepackage{url}
\usepackage{color}
\usepackage{dsfont}
\usepackage{bbm}
\usepackage{booktabs}
\usepackage{array}
\usepackage[table]{xcolor}
\usepackage{yfonts}
\usepackage{cleveref} 

\newtheorem{theorem}{Theorem}[section]

\newtheorem{corollary}[theorem]{Corollary}

\newtheorem{remark}{Remark}

\newtheorem{example}{Example}

\newcommand{\subscr}[2]{{#1}_{\textup{#2}}}

\newcommand{\until}[1]{\{1,\dots,#1\}}

\newcommand{\Ker}{\operatorname{Ker}}

\newcommand{\Image}{\operatorname{Im}}

\newcommand{\real}{\mathbb{R}}

\newcommand{\transpose}{\mathsf{T}} 

\newcommand{\mc}{\mathcal}

\DeclareSymbolFont{bbold}{U}{bbold}{m}{n}
\DeclareSymbolFontAlphabet{\mathbbold}{bbold}

\newcommand\oprocendsymbol{\hbox{$\square$}}
\newcommand\oprocend{\relax\ifmmode\else\unskip\hfill\fi\oprocendsymbol}

\renewcommand{\theenumi}{(\roman{enumi})}

\newcommand*{\QEDA}{\hfill\ensuremath{\blacksquare}}%

\renewcommand{\baselinestretch}{0.977}

\begin{document}
\title{\bf Learning Minimum-Energy Controls from Heterogeneous Data}
\author{Giacomo Baggio and Fabio Pasqualetti \thanks{This material is
    based upon work supported in part by awards ARO 71603NSYIP, ARO
    W911NF-18-1-0213, and AFOSR FA9550-19-1-0235. Giacomo Baggio is with the Department of Information Engineering, University of Padova, Italy,  e-mail: \href{mailto:baggio@dei.unipd.it}{\texttt{baggio@dei.unipd.it}}. 
    Fabio~Pasqualetti is with the Department of Mechanical
    Engineering, University of California at Riverside,  e-mail: \href{mailto:fabiopas@engr.ucr.edu}{\texttt{fabiopas@engr.ucr.edu.}}}
}
\maketitle


\begin{abstract} 
  In this paper we study the problem of learning minimum-energy
  controls for linear systems from heterogeneous data. Specifically,
  we consider datasets comprising input, initial and final state
  measurements collected using experiments with different time
  horizons and arbitrary initial conditions. In this setting, we first
  establish a general representation of input and sampled state
  trajectories of the system based on the available data. Then, we
  leverage this data-based representation to derive closed-form
  data-driven expressions of minimum-energy controls for a wide range
  of control horizons. Further, we characterize the minimum number of
  data required to reconstruct the minimum-energy inputs, and discuss
  the numerical properties of our expressions. Finally, we investigate
  the effect of noise on our data-driven formulas, and, in the case of
  noise with known second-order statistics, we provide corrected
  expressions that converge asymptotically to the
  true~optimal~control~inputs.
\end{abstract}

\section{Introduction}\label{sec: introduction}
The availability of large volumes of freely accessible data and the
recent advances in machine learning and artificial intelligence are
revolutionizing many areas of science and engineering. These include
control and system theory, in which direct data-driven control design
has recently been recognized as an appealing (and sometimes
preferable) alternative to the classic model-based paradigm
\cite{CDP-PT:19,JC-JL-FD:18,HJW-JE-HLT-MKC:19,JB-FA:19,GRGS-ASB-CL-LC:19,GB-VK-FP:19}. In
particular, learning controls directly from data turns out to be
beneficial when an accurate model of the system is difficult or
expensive to obtain from first principles, or when system
identification leads to significant errors or excessive computational
costs in the reconstruction of the~desired~control.

Several direct data-driven control design approaches have been proposed and
analyzed in the literature (see \cite{ZSH-ZW:13} for an overview of
recent results). These differ in the class of dynamics, control
objective, and data collection, and include, among others, (model-free)
reinforcement learning \cite{BR:18}, iterative learning
control \cite{DAB-MT-AGA:06}, adaptive control
\cite{KJA-BW:73}, and behavior- or subspace-based methods
\cite{IM-PR:08,GRGS-ASB-CL-LC:19,CDP-PT:19}.

In this paper, we focus on learning the minimum-energy control input
driving a linear system from an initial state to a desired target
one. We show that this control input can be exactly reconstructed from
data consisting of heterogeneous and, in certain cases, noisy
measurements of system trajectories. In particular, we establish
closed-form data-driven expressions of minimum-energy controls for
noiseless and noisy data. Besides further supporting the intriguing
idea that data-driven control represents a viable alternative to
model-based control, our framework and results offer a different,
attractive perspective on many problems in network analysis and
control. In fact, (model-based) minimum-energy controls have been
extensively employed for controlling, and characterizing the control
performance of, large-scale networks governed by linear dynamics,
e.g., see \cite{FP-SZ-FB:13q,THS-FLC-JL:16,NB-GB-SZ:17}.

\noindent \textbf{Related work.} 
{\color{black}The data-driven framework employed in this paper is similar to the one of \cite{CDP-PT:19,JB-FA:19,HJW-JE-HLT-MKC:19}, which can, in turn, be viewed as a state-space adaptation of the behavioral setting described in, e.g., \cite{JCW-PR-IM-BLMDM:05,IM-PR:08,JC-JL-FD:18}.
These works} exploit a data-based representation of the system in terms
of data that typically consist of uninterrupted samples of a single,
noiseless, and sufficiently long input-output trajectory. Here,
instead, we consider data collected from system trajectories with possibly
different time horizons and initial conditions. Further, under some
assumptions on the noise model, we {\color{black} establish asymptotic results for} case of data
corrupted by noise. Finally, besides our earlier work
\cite{GB-VK-FP:19}, we are not aware of data-driven approaches
tailored to minimum-energy~controls.


\noindent \textbf{Contribution.} The contributions of this paper are
threefold. First, we provide a data-based representation of sampled
system trajectories based on data comprising input, initial and final
state measurements collected via control experiments with different
time horizons, arbitrary inputs and initial conditions. Second, based
on these data, we establish two equivalent closed-form expressions of
the minimum-energy control input to reach a desired target
state. Differently from \cite{GB-VK-FP:19}, our expressions can be
used to compute minimum-energy controls for a wide range of control
times, and, in particular, for times that are determined
  only by the experimental data and that can exceed the largest time
  horizon of the collected experiments. Further, we discuss the numerical
properties of our data-driven expressions, and the minimum number of
data required to correctly reconstruct the minimum-energy control
inputs. Third and finally, in the case of data corrupted by noise with
known second-order statistics, we propose corrected data-driven
control expressions, and show that these converge to the true control
inputs in the limit of~infinite~data.


\noindent \textbf{Organization.} The rest of the paper is organized as
follows. In Section \ref{sec:sys-and-data}, we illustrate the class of systems  and data collection
setting considered in this paper. In Section \ref{sec:repres}, we establish a data-based parameterization
of sampled system trajectories. In Section \ref{sec:min-en} and
\ref{sec:noise}, we present and discuss data-driven expressions of
minimum-energy controls for the case of noiseless and noisy data,
respectively. Finally, Section \ref{sec:conc} contains some concluding remarks and future directions.

\noindent \textbf{Notation.} Given a matrix $A\in\real^{p\times q}$,
we let $\Ker(A)$ and $A^{\dag}$ denote the kernel and Moore--Penrose
pseudoinverse of $A$, respectively. We let $0_{n,m}$ and $I_n$ denote
the $n\times m$ zero matrix (we simply write $0_n$ if $m=n$) and
$n\times n$ identity matrix, respectively. We will omit the subscripts
when the dimensions are clear from the context. Further, we denote
with $K_{A}$ the matrix whose columns form a basis of~$\Ker(A)$.

\section{System dynamics and available data}\label{sec:sys-and-data}
Consider a discrete-time linear time-invariant system
\begin{align}\label{eq:sys}
  x(t+1)=Ax(t)+Bu(t), \quad x(0)=x_{0}\in\real^{n},
\end{align}
where $x(t)\in\real^{n}$ and $u(t)\in\real^{m}$ are the state and
input of the system at time $t$, and $A\in\real^{n\times n}$ and
$B\in\real^{n\times m}$ are the state and input matrices,
respectively. Let $C_{T}=[B \ \ AB \ \ \cdots \ \ A^{T-1} B\big]$
denote the $T$-steps controllability matrix of the system
\eqref{eq:sys}. We assume that $A$ and $B$ are unknown, and that a set
of control experiments with the system \eqref{eq:sys} has been
conducted for control purposes. Each control experiment consists of
(i) generating a $T$-steps input sequence
$u_{T}=[u(T-1)^{\transpose}, \dots, u(0)^{\transpose}]^{\transpose}
\in \real^{mT}$, and (ii) measuring the state of the system with input
$u_{T}$ at time $t=0$, namely $x(0)$, and at time $t=T$, namely,
\begin{align}
	x(T) = A^{T}x(0)+C_{T} u_{T}.
\end{align}
We assume that the control experiments have been performed using
$M$ distinct time horizons $T_{i}\in\mathbb{N}$, $i\in\until{M}$,
and we divide the available data in sets $(U_{i},X_{0,i},X_{i})$,
$i\in\until{M}$, {\color{black}where the $i$-th set contains $N_{i}$ experiments, and} $U_{i}\in\real^{mT_{i}\times N_{i}}$,
$X_{0,i}\in\real^{n\times N_{i}}$, and $X_{i}\in\real^{n\times N_{i}}$
denote the matrices whose columns contain, respectively, the input
sequences with horizon $T_i$, the initial states of the experiments, and the final state measurements recorded at time
$T_{i}$. We let $\mc D=\{(U_{i},X_{0,i},X_{i})\}_{i=1}^{M}$~denote~the~set~of~all~available~data.

We stress that, equivalently, $\mc D$ may comprise measurements that
have (intermittently) been recorded from a sufficiently long
experiment or from several short and independent ones (possibly
performed using different initializations). The first scenario is
quite standard for system identification \cite{PVO-BLDM:12}
and behavior-based control \cite{CDP-PT:19}, where
data typically consist of a single system trajectory (the case of
missing observations has been analyzed in a limited number of works, e.g., see \cite{IM:17}). The second experimental scenario has
recently been considered in \cite{SD-HM-NM-BR-ST:18,GB-VK-FP:19},
under the more restrictive assumption that the initial state is the same~for~all~experiments.  

\section{Data-based representation of sampled system trajectories}\label{sec:repres}
Consider a sequence of (possibly repeated) indices
$k_{1},\dots,k_{\ell}\in\{1,\dots, M\}$, and let
$T=\sum_{i=1}^{\ell} T_{k_{i}}$. Further,~let
\begin{align*}
  x_{k_{1},\dots,k_{\ell}}=\left[x(0)^{\transpose},
  x(T_{k_{1}})^{\transpose},
  x(T_{k_{1}}+T_{k_{2}})^{\transpose},\dots,x
  \left(T\right)^{\transpose} \right]^{\transpose}
\end{align*}
denote the state trajectory of \eqref{eq:sys} generated by the control
input $u_{T}\in\real^{mT}$ and sampled at times $0$,
$T_{k_{1}},T_{k_{1}}+T_{k_{2}},\dots,T$. For notational convenience,
we write $x_{0:T}$ when $T_{k_{i}}=1$ for all $i$. The next result
provides a parameterization of all admissible pairs
$(u_{T},x_{k_{1},\dots,k_{\ell}})$ in terms of the data $\mc D$.

 
\begin{theorem}{\emph{\bfseries(Data-based representation of input and
      sampled state pairs)}}\label{thm: data-based representation}
  If
  $[X_{0,k_{i}}^{\transpose} \ U_{k_{i}}^{\transpose} ]^{\transpose}$
  is full row rank for all $i\in\until{\ell}$, then any pair
  $(u_{T},x_{k_{1},\dots,k_{\ell}})$ of input and sampled state
  trajectories of the system \eqref{eq:sys} satisfies
\begin{align}\label{eq: data-based representation}
	\begin{bmatrix} u_{T} \\ x_{k_{1},\dots,k_{\ell}}  \end{bmatrix} = \begin{bmatrix} G \\ H \end{bmatrix}\alpha, \quad \alpha\in\real^{q_{k_{\ell}}+\dots q_{k_{1}}+n},
\end{align}
where $q_{k_i}=\dim\,\Ker(X_{0,k_i})$, and
\begin{align}
	G & = \begin{bmatrix} \scriptstyle \tilde U_{{\ell}} & \scriptstyle 0 & \cdots & \scriptstyle 0 & \scriptstyle 0_{n} \\
	\scriptstyle 0 & \scriptstyle \tilde U_{{\ell-1}}  &  \ddots & \vdots & \vdots \\
	\vdots & \ddots &  \ddots & \scriptstyle  0 &  \scriptstyle 0_{n} \\
	\scriptstyle 0 & \cdots &\scriptstyle 0 &\scriptstyle \tilde U_{{1}} &\scriptstyle 0_{n}
\end{bmatrix}, \label{eq:G} \\
	H &  = \begin{bmatrix} \scriptstyle 0 &  \cdots & \scriptstyle 0 & \scriptstyle 0 & \scriptstyle I \\[0.1cm]
	\scriptstyle 0  &  \cdots  & \scriptstyle 0 & \scriptstyle \tilde X_{{1}}  &  \scriptstyle Q_{1}\\
	  \vdots & \scriptstyle \reflectbox{$\ddots$}  & \scriptstyle \tilde X_{{2}} & \scriptstyle Q_{2} \tilde X_{{1}}  & \scriptstyle Q_{2}Q_{1} \\ 
	\scriptstyle 0 & \scriptstyle \reflectbox{$\ddots$}  & \scriptstyle \vdots &  \scriptstyle \vdots &  \scriptstyle \vdots \\
	\scriptstyle \tilde X_{{\ell}}   & \cdots & \scriptstyle \prod\limits_{i=0}^{\ell-3}Q_{\ell-i} \tilde X_{{2}}  & \scriptstyle \prod\limits_{i=0}^{\ell-2}Q_{\ell-i}\tilde X_{{1}} &\!\!\! \scriptstyle \prod\limits_{i=0}^{\ell-1}Q_{\ell-i} \end{bmatrix}, \label{eq:H}
\end{align}
with $\tilde{U}_{i} = U_{k_i}K_{X_{0,k_{i}}}$, $\tilde{X}_{i} = X_{k_{i}}K_{X_{0,k_i}}$, 
and $Q_{i}=X_{k_{i}}K_{U_{k_{i}}}(X_{0,k_{i}}K_{U_{k_{i}}})^{\dag}$, for all $i\in\until{\ell}$.
\end{theorem}\smallskip


\begin{proof}
  Note that, since
  $[X_{0,k_{i}}^{\transpose} \ U_{k_{i}}^{\transpose}]^{\transpose}$
  is full row rank for all $i\in\until{\ell}$,
  $\tilde{U}_{i}=U_{k_{i}}K_{X_{0,k_{i}}}$ is full row rank for all
  $i\in\until{\ell}$.\footnote{Indeed, since
    \smash{$[X_{0,k_{i}}^{\transpose} \ U_{k_{i}}^{\transpose}]^{\transpose}$}
    is full row rank, for all $u\in\real^{mT}$ there exists
    $\gamma\in\Ker(X_{0,k_{i}})$
    such that \smash{$[0 \ u^{\transpose}]^{\transpose} =
    [X_{0,k_{i}}^{\transpose} \
    U_{k_{i}}^{\transpose}]^{\transpose}\gamma$}, which implies that
    $U_{k_{i}}K_{X_{0},k_{i}}$ must be of full row~rank. } From
  \eqref{eq:G}, this implies that $G$ is full row rank, and,
  therefore, for every $T$-steps input sequence $u_{T}$ there exists a
  real vector $\alpha$ such that $u_{T}=G\alpha$. We next show that
  the sampled state $x_{k_{1},\dots,k_{\ell}}$ corresponding to the input $u_{T}=G\alpha$ can be expressed as $H\alpha$, with $H$
  as in \eqref{eq:H}. To this aim, let $C_{T_{i}}$ denote the
  $T_{i}$-steps controllability matrix of \eqref{eq:sys}, and observe
  that, for all $j\in\until{\ell}$,
  \begin{align}\label{eq:sampled evolution 1}
  	&x(T_{k_{1}}+\cdots+T_{k_{j}}) = A^{T_{k_{1}}+\cdots+T_{k_{j}}}x_0 +  \notag\\
	& \hspace{1cm} +A^{T_{k_{2}}+\cdots+T_{k_{j}}}C_{T_{k_{1}}} \tilde U_{1} \alpha_{1}+\cdots+C_{T_{k_{j}}}\tilde U_{j}\alpha_{j},
  \end{align}
   where we partitioned $\alpha$ as $\alpha=[\alpha_{\ell}^{\transpose},\alpha_{\ell-1}^{\transpose},\dots,\alpha_{1}^{\transpose},\alpha_{0}^{\transpose}]^{\transpose}$, with $\alpha_{i}\in\real^{q_{k_{i}}}$, and $\alpha_{0}\in\real^{n}$.
   Set $\alpha_{0}=x_0$. From 
   \begin{align*}
     \tilde{X}_{i} = X_{k_i}K_{X_0,k_i}	&= (A^{T_{k_i}}X_{0,k_i} + C_{T_{k_i}}U_{k_i})K_{X_0,k_i} \\
     			&= C_{T_{k_i}}U_{k_i}K_{X_0,k_i} = C_{T_{k_i}}\tilde U_{i},
   \end{align*}
   it follows that \eqref{eq:sampled evolution 1} can be rewritten as
  \begin{align}\label{eq:sampled evolution}
  	&x(T_{k_{1}}+\cdots+T_{k_{j}}) = A^{T_{k_{1}}+\cdots+T_{k_{j}}}\alpha_{0} \ + \notag\\
	&\hspace{2cm}+ A^{T_{k_{2}}+\cdots+T_{k_{j}}}\tilde X_{1} \alpha_{1}+\cdots+\tilde X_{j}\alpha_{j}.
  \end{align}
  Additionally, because $[X_{0,k_{i}}^{\transpose} \  U_{k_{i}}^{\transpose}]^{\transpose}$ is full row rank, $X_{0,k_{i}}K_{U_{k_{i}}}$ is full row rank, and from
  \begin{align*}
  	 X_{k_{i}}K_{U_{k_{i}}} &= (A^{T_{k_{i}}}X_{0,k_{i}} + C_{T_{k_{i}}}U_{k_{i}})K_{U_{k_{i}}} \\
	 &= A^{T_{k_{i}}}X_{0,k_{i}}K_{U_{k_{i}}}, 
  \end{align*}
  it follows that 
  \begin{align}\label{eq:Qk} 
  Q_{i}=X_{k_{i}}K_{U_{k_{i}}}(X_{0,k_{i}}K_{U_{k_{i}}})^{\dag}= A^{T_{k_{i}}}. 
  \end{align}
  Finally, by substituting \eqref{eq:Qk} into \eqref{eq:sampled evolution} and rewriting the latter in vector form, we obtain $x_{k_{1},\dots,k_{\ell}}=H\alpha$, with $H$
  as in~\eqref{eq:H}.~ 
\end{proof}


The previous result states that any $T$-steps input sequence and
corresponding state trajectory sampled at times $0$,
$T_{k_{1}},T_{k_{1}}+T_{k_{2}},\dots,T$ of the system \eqref{eq:sys}
can be written as a linear combination of the columns of a matrix that
depends on the dataset $\mc D$ only. Intuitively, this sampled
data-based representation is obtained by suitably ``gluing'' together
the data-based representations of system trajectories of lengths
$T_{k_1},T_{k_2},\dots,T_{k_\ell}$. One of the advantages of our
parameterization is that it
provides a data-based description of a linear system that does not
rely on the identification of the system matrices $A$ and $B$. {\color{black} Further, when the full state of the system is accessible, the data-based
  representation of Theorem \ref{thm:
    data-based representation} generalizes those employed in a number of recent works (e.g., \cite{CDP-PT:19,GRGS-ASB-CL-LC:19,JC-JL-FD:18}), which rely on measurements of a single, uninterrupted, and sufficiently long input-output trajectory.\footnote{A partial extension of this setting to
    multiple measured trajectories has been proposed in
    \cite{IM-JCW-PR-BLMD:05,JB-FA:19}, under the rather restrictive
    assumption that these trajectories align over a sufficiently long
    window at their intersection.}}
  To clarify the notation and implications of Theorem \ref{thm: data-based
  representation}, we next illustrate our result by means of~a~simple~example.

\begin{example}{\emph{\bfseries (Illustration of Theorem \ref{thm:
        data-based representation})}}
  Consider the scalar system
  \begin{align}\label{eq:scalar sys}
    x(t+1) = a x(t) + u(t), \quad a\in\real,
  \end{align}
  and assume that $M=1$, $N_1=3$, $T_1=2$, that is, data have been
  generated from three control experiments performed using a single
  time horizon of length two. Further, consider the following dataset
  $\mc D = \{(U_1, X_{0,1}, X_1)\}$, where
  \begin{align*}
    U_1 = 
    \begin{bmatrix} 
      0 & 1 & 0 \\ 0 & 0 & 1 
    \end{bmatrix}, 
                           \ X_{0,1} = 
                           \begin{bmatrix} 
                             1 & 0 & 0 
                           \end{bmatrix}, 
                                     \ X_1 = 
                                     \begin{bmatrix} 
                                       a^2 & 1 & a 
                                     \end{bmatrix}. 
  \end{align*}
  Notice that
  $[X_{0,1}^{\transpose} \ U_{1}^{\transpose} ]^{\transpose}$ has full
  row rank, and that
  \begin{align*}
  K_{U_{1}} =
    \begin{bmatrix} 
      1 \\ 0 \\ 0 
    \end{bmatrix},
    \  K_{X_{0,1}}=
    \begin{bmatrix} 
      0 & 0 \\ 1 & 0 \\ 0 & 1 
    \end{bmatrix},\  Q_1=a^2.
  \end{align*}
  Thus, by choosing $\ell=2$ and $k_1=k_2=1$, by Theorem \eqref{thm:
    data-based representation}, any input $u_T$ and resulting state
  sampled at time $0$, $T_1=2$, $T=2T_1=4$, $x_{0,2,4}$, of
  \eqref{eq:scalar sys} satisfy \eqref{eq: data-based
    representation},~where
  \begin{align*}
    G  =\!\left[
    \begin{array}{cc|cc|c} 
      1 & 0 & 0 & 0 & 0 \\ 
      0 & 1 & 0 & 0 & 0 \\ 
      \hline 0 & 0 & 1 & 0 & 0 \\ 
      0 & 0 & 0 & 1 & 0 
    \end{array}\right]\!,
                      \ H  = \! \left[
                      \begin{array}{cc|cc|c} 
                        0 & 0 & 0 \!\! &\!\! 0\! & 1 \\ 
                        \hline 0 & 0 & 1 \!\!& \!\!a \!& a^2 \\ 
                        \hline  1 & a &  a^2 \!&\! \!a^3\! & a^4 
                      \end{array}\right]\!.
  \end{align*}
  We note, in particular, that to compute the matrices $G$ and $H$, we
  did not reconstruct the system parameter $a$.~\oprocend
\end{example}


When the dataset $\mc D$ contains trajectories recorded using a
unit-length time horizon\footnote{We remark that a unit-length dataset
  can be constructed from measurements of a single trajectory by dividing the latter
  into unit-length~segments.}, we have the following immediate
corollary of Theorem \ref{thm: data-based representation}, which
provides a complete data-based parameterization of all input sequences
and corresponding state trajectories of the system \eqref{eq:sys}.

\begin{corollary} {\emph{\bfseries(Complete data-based representation
      of input and state pairs)}} \label{cor: representation} Assume
  that there exists an index $j\in\until{M}$ such that $T_{j}=1$.
  If
  $[X_{0,j}^{\transpose} \ U_{j}^{\transpose} ]^{\transpose}$
  is full row rank, then, for any $T\ge 1$, any pair of input $u_T$
  and corresponding state trajectory $x_{0:T}$ of the system
  \eqref{eq:sys} satisfies
\begin{align}\label{eq: data-based representation full}
	\begin{bmatrix} u_{T} \\ x_{0:T}  \end{bmatrix} = \begin{bmatrix} G \\ H \end{bmatrix}\alpha, \quad \alpha\in\real^{Tq_{j}+n},
\end{align}
where $G$ and $H$ are defined as in \eqref{eq:G} and \eqref{eq:H}, respectively, with $\ell=T$ and $k_{i}=j$ for all $i\in\until{\ell}$. 
\end{corollary}

\section{Closed-form data-driven expressions of minimum-energy controls}\label{sec:min-en}



\subsection{Problem formulation}
For a control horizon $T\ge 1$ and desired initial and final states
$\subscr{x}{0}\in\real^n$ and $\subscr{x}{f}\in\real^n$, respectively,
the minimum-energy control problem asks for the input sequence $u_T\in\real^{mT}$
with minimum norm that steers the state of the system \eqref{eq:sys}
from $x_{0}$ to $\subscr{x}{f}$ in $T$ steps. Mathematically, this is
encoded in the solution of the following minimization problem:
\begin{align}\label{eq: min energy problem 1}
  \begin{array}{ll}
    \min\limits_{u_T} & \| u_T\|_2^2 ,\\[1em]
    \,\text{s.t.} 
         & x(t+1) = A x(t) + B u(t), \\[.5em]
         & x(0) =x_{0}, \ x(T) = \subscr{x}{f} .
  \end{array}
\end{align}
As a classic result \cite{TK:80}, the minimization problem \eqref{eq:
  min energy problem 1} is feasible if and only if $\subscr{x}{f}$ is
reachable in $T$-steps from $x_0$, or, equivalently, if and only if
$(\subscr{x}{f} - A^T x_0) \in \Image (C_T)$, where $C_T$ is the
$T$-steps controllability matrix of the system.  In this case, the
solution to \eqref{eq: min energy problem 1} is unique and can be
computed as
\begin{align}\label{eq: min energy input}
  u^*_T = C_T^\dag (\subscr{x}{f} - A^T x_0).
\end{align}
In the remaining of this section, we will derive closed-form
expressions of $u^*_T$ based on the dataset $\mc D$ without relying on
the identification of the system matrices $A$ and $B$. To this end, we
will make use of the following~assumptions:
\begin{enumerate}
\item[(A1)] The state $\subscr{x}{f}$ is reachable in $T$-steps from
  the state $x_0$.

\item[(A2)] The dataset $\mc D$ contains (possibly repeated) indices
  $k_{1},\dots,k_{\ell} \in \{1,\dots, M\}$ such that
  $\sum_{i=1}^{\ell} \!T_{k_{i}} = T$.
\end{enumerate}


\subsection{Data-driven expressions of minimum energy controls}
Let $k_{1},\dots,k_{\ell}\in\{1,\dots, M\}$ be such that
$\sum_{i=1}^{\ell} T_{k_{i}}=T$, and consider the following
minimization problem:
\begin{align}\label{eq: min energy problem}
  \begin{array}{lcl}
    &\min\limits_{\alpha} &  \| G \alpha \|_2^2 \\[1em]
    & \text{s.t.} 
                         & \begin{bmatrix} x_{0} \\ \subscr{x}{f}  \end{bmatrix} = \bar{H} \alpha,  
  \end{array}
\end{align}
where $\alpha\in \real^{q_{k_{\ell}}+\cdots+q_{k_{1}}+n}$ is the
optimization variable, $q_{k_i}=\dim\,\Ker(X_{0,k_i})$, $G$ is as in
\eqref{eq:G}, and $\bar H$ is the matrix comprising the first and last
(row) block of $H$ in \eqref{eq:H}, namely:
\begin{align}
  \bar H & = \begin{bmatrix} \scriptstyle 0 &  \cdots & \scriptstyle 0 & \scriptstyle 0& \scriptstyle I \\
	\scriptstyle \tilde X_{{\ell}}   & \cdots & \scriptstyle \prod\limits_{i=0}^{\ell-3}Q_{\ell-i} \tilde X_{{2}}  & \scriptstyle \prod\limits_{i=0}^{\ell-2}Q_{\ell-i}\tilde X_{{1}} &\!\!\! \scriptstyle \prod\limits_{i=0}^{\ell-1}Q_{\ell-i} \end{bmatrix}\!. \label{eq:bar H}
\end{align}
The next theorem shows that the solution to \eqref{eq: min energy
  problem} leads to a data-driven expression of the $T$-steps
minimum-energy control input from $x_0$ to $\subscr{x}{f}$ for the
system \eqref{eq:sys}.

\begin{theorem}{\emph{\bfseries(Data-driven minimum-energy
      controls)}}\label{thm: optimal controls}
  Assume that
  $[X_{0,k_{i}}^{\transpose} \ U_{k_{i}}^{\transpose}]^{\transpose}$
  is full row rank for all $i\in\until{\ell}$. The $T$-steps minimum-energy
  control input to drive the system \eqref{eq:sys} from $x_{0}$ to
  $\subscr{x}{f}$ can be expressed as
\begin{align}\label{eq: u data driven}
	u_T^{*} = (I - G K_{\bar H} (GK_{\bar H})^{\dag}) G{\bar H}^{\dag}\begin{bmatrix} x_{0} \\ \subscr{x}{f}  \end{bmatrix}.
\end{align}\smallskip
\end{theorem}
\begin{proof}
Since $[X_{0,k_{i}}^{\transpose} \  U_{k_{i}}^{\transpose}]^{\transpose}$ has full row rank for all $i\in\until{\ell}$ and $\subscr{x}{f}$ is reachable in $T$ steps from $x_{0}$ by assumption, Theorem \ref{thm: data-based representation} ensures that there exists a real vector $\alpha^{*}$ satisfying 
\begin{align*}
u_T^{*}=G\alpha^{*}\ \   \text{ and }\ \  \begin{bmatrix} x_{0} \\ \subscr{x}{f}  \end{bmatrix} = \bar{H} \alpha^{*}.
\end{align*}
Because the $T$-steps minimum-energy control input $u_T^{*}=G\alpha^{*}$ is unique, $\alpha^*$ is also a solution to problem \eqref{eq: min energy problem}, and its computation is equivalent to computing $u^*_T$. By direct calculation, any solution to problem \eqref{eq: min energy problem}~has~the~form
\[
	\alpha^{*}=(\bar H^{\dag} -  K_{\bar H} (GK_{\bar H})^{\dag}G{\bar H}^{\dag}) \begin{bmatrix} x_{0} \\ \subscr{x}{f}  \end{bmatrix} +g,
\] 
where $g$ is an arbitrary vector belonging to the kernel of $G$. 
Finally, by substituting the above expression of $\alpha^*$ in $u_T^{*}=G\alpha^{*}$, the data-driven expression \eqref{eq: u data driven} directly follows.~ 
\end{proof}

Theorem \ref{thm: optimal controls} exploits the solution to the optimization problem \eqref{eq: min energy problem} and the data-based representation of sampled system trajectories established in Theorem \ref{thm: data-based representation} to compute a closed-form data-driven expression of the minimum-energy input $u^*_T$ based on the dataset $\mc D$. Alternatively, a data-based expression of $u_T^*$ can be derived via estimation of the $T$-steps controllability matrix $C_T$ and matrix $A^T$, as we show~next.


\begin{theorem}{\emph{\bfseries(Alternative expression of data-driven
      minimum-energy controls)}}\label{thm: alternative optimal
    controls}
  Assume that
  $[X_{0,k_{i}}^{\transpose} \ U_{k_{i}}^{\transpose}]^{\transpose}$
  is full row rank for all $i\in\until{\ell}$. The $T$-steps minimum-energy
  input to drive \eqref{eq:sys} from $x_{0}$ to
  $\subscr{x}{f}$ can be expressed~as
\begin{align}\label{eq: alternative u data driven}
  u_T^{*} =  \hat{C}_{T}^{\dag}\begin{bmatrix}  -\prod\limits_{i=0}^{\ell-1}Q_{\ell-i}  & I \end{bmatrix} \begin{bmatrix} x_{0} \\ \subscr{x}{f}  \end{bmatrix},
\end{align}
where, for all $i\in\until{\ell}$,
\begin{align} \label{eq:Qi,Li}
  \begin{split}
    \hat C_{T} &= 
    \begin{bmatrix} L _{\ell}& Q_{\ell} L_{\ell-1}  & \cdots &
      \prod\limits_{i=0}^{\ell-2}Q_{\ell-i}L_{1} \end{bmatrix},\\
    Q_i &= X_{k_{i}}K_{U_{k_{i}}}(X_{0,k_{i}}K_{U_{k_{i}}})^{\dag},
    \text{ and }\\[0.125cm]
    L_{i} &= X_{k_{i}}K_{X_{0,k_{i}}}(U_{k_{i}}K_{X_{0,k_{i}}})^{\dag}.
  \end{split}
\end{align}

\end{theorem}\medskip
\begin{proof} Notice that
   \begin{align*}
     X_{k_{i}}K_{U_{k_{i}}} &= (A^{T_{k_{i}}}X_{0,k_{i}} + C_{T_{k_{i}}}U_{k_{i}})K_{U_{k_{i}}} \\
	 &= A^{T_{k_{i}}}X_{0,k_{i}}K_{U_{k_{i}}}. 
   \end{align*}
   Because $[X_{0,k_{i}}^{\transpose} \  U_{k_{i}}^{\transpose}]^{\transpose}$ has full row rank for all $i$, $X_{0,k_{i}}K_{U_{k_{i}}}$ has also full row rank for all $i$, so that it holds 
   \begin{align}\label{eq:Qi}
     Q_{i} &= X_{k_{i}}K_{U_{k_{i}}}(X_{0,k_{i}}K_{U_{k_{i}}})^{\dag}= A^{T_{k_{i}}}.  
   \end{align}
   Similarly, notice that
   \begin{align*}
  	  X_{k_i}K_{X_0,k_i}	&= (A^{T_{k_i}}X_{0,k_i} + C_{T_{k_i}}U_{k_i})K_{X_0,k_i}, \\
     			&= C_{T_{k_i}}U_{k_i}K_{X_0,k_i},
  \end{align*}
  and, because $U_{k_{i}}K_{X_{0,k_{i}}}$ has full row rank for all $i$, we have
  \begin{align}\label{eq:Li}
     L_{i} &= X_{k_{i}}K_{X_{0,k_{i}}}(U_{k_{i}}K_{X_{0,k_{i}}})^{\dag}= C_{T_{i}}. 
   \end{align}
   From \eqref{eq:Qi} and \eqref{eq:Li}, it follows that $\hat C_{T} =  C_{T}$ and $\prod_{i=0}^{\ell-1}Q_{\ell-i}=A^{T}$. Finally, since, by assumption, $\subscr{x}{f}$ is reachable in $T$ steps from $x_{0}$, the data-driven expression \eqref{eq: alternative u data driven} directly follows from the model-based expression \eqref{eq: min energy input}. 
\end{proof}

\begin{figure}
  \vspace{0.15cm}
  \includegraphics[width=.485\textwidth]{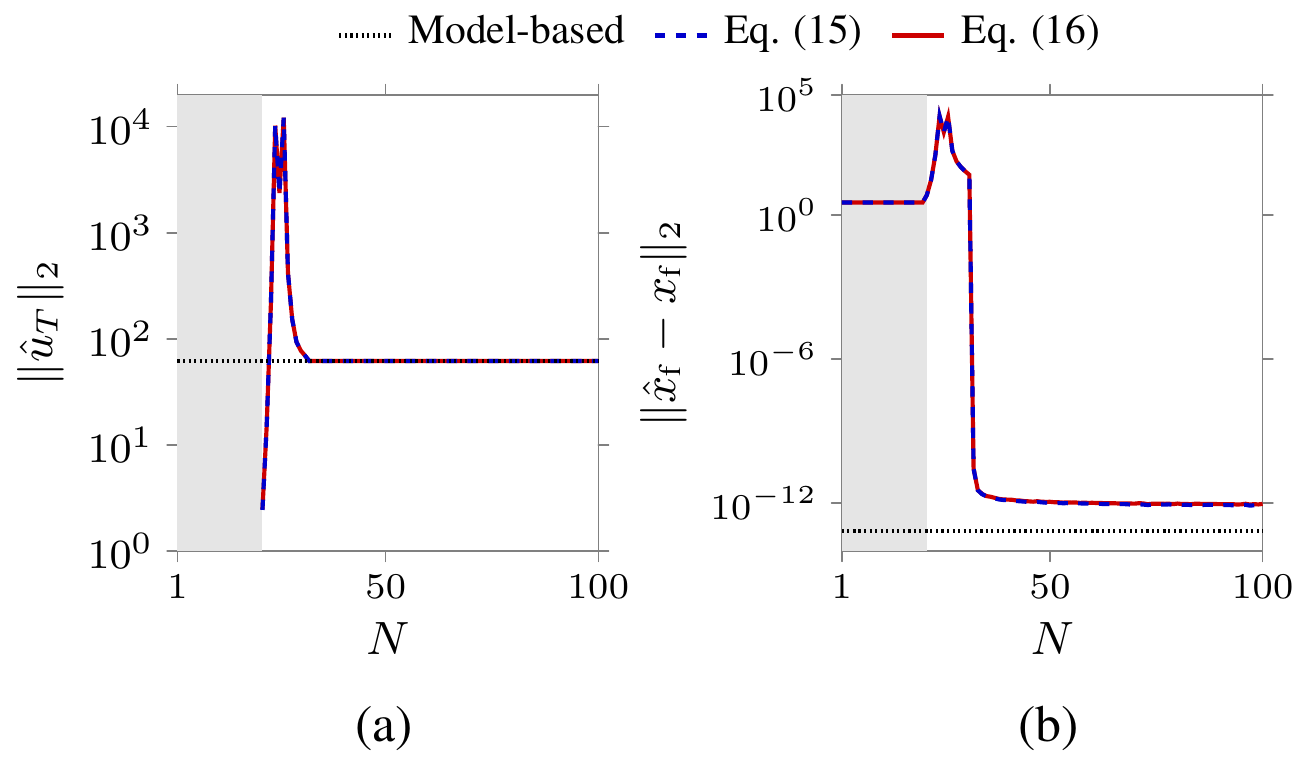}
  \vspace{-0.5cm}
  \caption{In this figure we compare the numerical performance of the model-based minimum-energy input \eqref{eq: min energy input} and the data-driven minimum-energy inputs \eqref{eq: u data driven} and \eqref{eq: alternative u data driven}. We choose a system of dimension $n=20$ with $m=2$ inputs. The system matrices $A$ and $B$ have been generated randomly with i.i.d.~normal entries.  Data have been divided into  $M=4$ datasets with time horizons $\smash{T_{i}=2+i}$, $i=1,\dots,M$. The $i$-th dataset, $i=1,\dots,M$, contains $N_{i}=N$ measurements. We choose a control horizon $\smash{T=\sum_{i=1}^{M}T_{i}=18}$.  The entries of the data matrices $X_{0,i}$ and $U_{i}$, initial state $x_{0}$, and final state $\subscr{x}{f}$ have been independently drawn from of a normal distribution. 
 The plots show the norm of the minimum-energy input (panel (a)) and the corresponding error in the final state (panel (b)) for the model-based expression \eqref{eq: min energy input} and the data-driven expressions \eqref{eq: u data driven} and \eqref{eq: alternative u data driven} as a function of the size of the datasets $N$. For the data-driven input \eqref{eq: u data driven} we replace $\smash{(G K_{\bar H})^\dag}$ with $\smash{(G K_{\bar H})^{\dag}_{\varepsilon}}$ where $\varepsilon=10^{-8}$ (cf.~Remark \ref{rem:numerics}). All curves concerning the data-driven strategies represent the average over 500 random realizations of the data matrices. In the gray regions the data-driven inputs are zero since the kernel of every matrix $X_{0,i}$ and~$U_{i}$~is~empty.} 
  \label{fig:numerical analysis}
\end{figure}

In Fig.~\ref{fig:numerical analysis} we compare the numerical
performance of the model-based input \eqref{eq: min energy input} and
our data-driven expressions \eqref{eq: u data driven} and \eqref{eq:
  alternative u data driven} for a system of dimension $n=20$, a number
of inputs $m=2$, and randomly generated data consisting of $M=4$
datasets featuring different time horizons. Each dataset contains an
identical number of data $N$. For values of $N$ in the gray region,
the kernel of every data matrix $X_{0,i}$ and $U_{i}$ is empty and,
therefore, the data-driven inputs \eqref{eq: u data driven} and
\eqref{eq: alternative u data driven} are zero. As soon as $N$ equals
the number of rows of the largest matrix
$[X_{0,k_{i}}^{\transpose} \ U_{k_{i}}^{\transpose}]^{\transpose}$
($N=32$ in the figure), the norm of the data-driven inputs reaches the
optimal one (Fig.~\ref{fig:numerical analysis}(a)), and the corresponding error in the final state rapidly
decays to zero (Fig.~\ref{fig:numerical analysis}(b)), in agreement with Theorems \ref{thm: optimal
  controls}~and~\ref{thm: alternative optimal controls}.


\begin{remark}{\emph{\bfseries(Minimum number of required
      experiments)}} Theorems \ref{thm: optimal controls} and
  \ref{thm: alternative optimal controls} provide exact data-driven
  expressions of the $T$-steps minimum-energy control input from
  $x_{0}$ to $\subscr{x}{f}$, under the assumption that the data
  matrix
  $[X_{0,k_{i}}^{\transpose} \ U_{k_{i}}^{\transpose}]^{\transpose}$
  is full row rank for all $i\in\until{\ell}$. For this condition to
  be satisfied, at least $N_{i}= T_{k_{i}}m+n$ experiments must be
  collected for each control time $T_{k_{i}}$. If there exists
  $j\in\until{M}$ such that $T_{k_{j}}=1$ (unit-length data), $m+n$
  measurements suffice to reconstruct the $T$-steps minimum-energy
  control input, for every horizon $T$. In this case, our expressions
  implicitly estimate the system matrices $A$ and $B$. Specifically,
  in \eqref{eq: u data driven} and \eqref{eq: alternative u data
    driven}, $Q_j=A$, and, in \eqref{eq: alternative u data driven},
  $L_j=B$. Hence, in this case, using our data-driven expressions or a
  sequential system identification and control design approach seem to
  be equivalent from a computational viewpoint.~\oprocend
\end{remark}


\begin{remark}{\emph{\bfseries(Numerical properties of  \eqref{eq: u data driven} and \eqref{eq: alternative u data driven})}}\label{rem:numerics}
  While the data-driven expression \eqref{eq: alternative u data
    driven} appears to be numerically stable (i.e., small numerical errors yield small deviations from the minimum-energy control), \eqref{eq: u data
    driven} suffers from numerical instabilities. Precisely, in the
  case of small numerical errors, the (row) rank of matrix $G K_{\bar H}$
  could become full, yielding $u_T^{*}=0$ in \eqref{eq: u data driven}
  regardless of the value of $x_0$ and $\subscr{x}{f}$. To remedy this
  situation, it is numerically convenient to replace
  $(G K_{\bar H})^\dag$ in \eqref{eq: u data driven} with
  $(G K_{\bar H})^{\dag}_{\varepsilon}$, where
  $(A)^{\dag}_{\varepsilon}$ denotes the Moore--Penrose pseudoinverse
  of $A$ that treats as zero the singular values of $A$ that are
  smaller than $\varepsilon>0$. 
  As a rule of thumb, $\varepsilon$ should be set to a value
  slightly larger than the expected magnitude of the
  numerical~errors.~\oprocend
\end{remark}

\section{Data-driven minimum-energy controls\\ with noisy data}\label{sec:noise}
In this section, we assume that the dataset $\mc D$ is corrupted by
additive i.i.d.~noise with known second-order
statistics. Specifically, for all $i\in\until{M}$, we consider
corrupted data matrices of the form
\begin{align}\label{eq:corrupted data}
  \begin{split}
    U_{i} &= \bar U_{i}+W_{U_{i}}, 
    X_{0,i} = \bar X_{0,i}+W_{X_{0,i}}, 
    X_{i} = \bar X_{i}+W_{X_{i}},
  \end{split}
\end{align}
where $\bar U_{i}$, $\bar X_{0,i}$, and $\bar X_{i}$ denote the true
data matrices, and the entries of $W_{U_{i}}$, $W_{X_{0,i}}$, and
$W_{X_{i}}$ are i.i.d.~random variables with zero mean and variance
$\sigma^{2}_{U_{i}}$, $\sigma^{2}_{X_{0,i}}$, and
$\sigma^{2}_{X_{i}}$, respectively. In this case, the data-driven
expressions \eqref{eq: u data driven} and \eqref{eq: alternative u
  data driven} are typically biased (see \cite[Remark 3]{GB-VK-FP:19}
for an explicit example in a simplified scenario), yielding incorrect
control inputs even when the number of data grows unbounded. In this
section, we will show that the effect of noise can be cancelled, in
the limit of infinite data, by suitably ``correcting'' these
expressions. Specifically, inspired by \cite{IM-RJV-SVH:06}, we will introduce correction terms that
compensate for the variance-dependent terms generated by the
pseudoinverse and kernel operations in \eqref{eq: u data driven} and
\eqref{eq: alternative u data driven}, leading to asymptotically
correct (or, equivalently, \emph{consistent}) data-driven
expressions.\footnote{To simplify the treatment without compromising
  the generality of the approach, in what follows we will assume
  $N_{i}=N$, $\smash{\sigma_{U}^{2}=\sigma^{2}_{U_{i}}}$,
  $\smash{\sigma_{X_{0}}^{2}=\sigma^{2}_{X_{0,i}}}$, and
  $\smash{\sigma_{X}^{2}=\sigma^{2}_{X_{i}}}$ for all
  $i\in\until{M}$.}

We consider first the data-driven expression \eqref{eq: alternative u
  data driven}, and rewrite the terms $Q_i$, $L_i$ in
\eqref{eq:Qi,Li}, respectively, as
\begin{align*}
  Q_i & = X_{k_i} \Pi_{U_{k_i}}  X_{0,k_i}^\transpose (X_{0,k_i} \Pi_{U_{k_i}}   X_{0,k_i}^\transpose)^{\dagger}, \\
  L_i & = X_{k_i}\Pi_{X_{0}} U_{k_i}^\transpose (U_{k_i}\Pi_{X_{0}}  U_{k_i}^\transpose )^{\dagger},
\end{align*}
where we used the identity
$A^{\dag}=A^{\transpose}(A A^{\transpose})^{\dag}$, and we replaced,
without loss of generality, every term $K_A K_A^\transpose$ with the
orthogonal projections onto $\Ker(A)$, $\Pi_A=I-A^\dag A$. Next, we
define the ``corrected'' versions of $Q_i$ and $L_i$ as
\begin{align*}
  Q_{i,c} & = X_{k_i} \Pi_{U_{k_i},c}  X_{0,k_i}^\transpose (X_{0,k_i} \Pi_{U_{k_i},c}   X_{0,k_i}^\transpose-N\sigma_{X_{0}}^{2}I)^{\dagger}, \\
  L_{i,c} & = X_{k_i}\Pi_{X_{0},c} U_{k_i}^\transpose (U_{k_i}\Pi_{X_{0},c}  U_{k_i}^\transpose -N\sigma_{U}^{2}I)^{\dagger},
\end{align*}
where
\smash{$\Pi_{X_{0,k_i},c} = I-X_{0,k_{i}}^{\transpose}(X_{0,k_{i}}X_{0,k_{i}}^{\transpose}-N\sigma_{X_{0}}^{2}I)^{\dag}X_{0,k_{i}}$}
and
\smash{$\Pi_{U_{k_i},c}  = I-U_{k_{i}}^{\transpose}(U_{k_{i}}U_{k_{i}}^{\transpose}-N\sigma_{U}^{2}I)^{\dag}U_{k_{i}}$}. With
these definitions in place, we introduce the following ``corrected''
expression of the data-driven control input \eqref{eq: alternative u
  data driven}:
\begin{align}\label{eq: alt u data driven corr}
	u_{T,c}'' = \hat{C}_{T,c}^{\dag}\begin{bmatrix}  -\prod\limits_{i=0}^{\ell-1}Q_{\ell-i,c}  & I \end{bmatrix}\begin{bmatrix} x_{0} \\ \subscr{x}{f} \end{bmatrix},
\end{align}
where $\hat C_{T,c}$ is defined as in \eqref{eq:Qi,Li}, after replacing all instances of $Q_i$ and $L_i$ with $Q_{i,c}$ and $L_{i,c}$, respectively. It is worth noting that, if only the matrices $X_{i}$ are affected by noise, then \eqref{eq: alt u data driven corr} coincides with \eqref{eq: alternative u data driven}, and no correction is~needed.
\begin{theorem}{\emph{\bfseries(Consistency of $u_{T,c}''$)}}\label{thm: consistency} Assume that the dataset $\mc D$ is corrupted by noise as in \eqref{eq:corrupted data}, and that $[\bar X_{0,k_{i}}^{\transpose} \ \bar U_{k_{i}}^{\transpose}]^{\transpose}$ is full row rank for all $i\in\until{\ell}$. The data-driven control $u_{T,c}''$ in \eqref{eq: alt u data driven corr} converges almost surely to the minimum-energy control input $u^*_T$ as $N\to\infty$.
\end{theorem}
\begin{proof} By the Strong Law of Large Numbers \cite[p.~6]{AWV00} and the assumption on the noise, as $N\to \infty$, we have
\begin{align}\label{eq:LLN}
\begin{split}
	&\Delta_{i,1}\!=\!\frac{1}{N} X_{0,k_{i}}X_{0,k_{i}}^{\transpose} \xrightarrow{\text{a.s.}}  \frac{1}{N}\bar{X}_{0,k_{i}}\bar{X}_{0,k_{i}}^{\transpose}\!+\!\sigma_{X_{0}}^{2}I=\bar\Delta_{i,1}, \\&
	\Delta_{i,2}\!=\!\frac{1}{N} U_{k_{i}}U_{k_{i}}^{\transpose} \xrightarrow{\text{a.s.}}  \frac{1}{N} \bar U_{k_{i}}\bar U_{k_{i}}^{\transpose} +\sigma_{U}^{2}I=\bar\Delta_{i,2}, \\
	&\Delta_{i,3}\!=\!\frac{1}{N}X_{k_{i}}X_{0,k_{i}}^{\transpose} \xrightarrow{\text{a.s.}}  \frac{1}{N}\bar X_{k_{i}}\bar X_{0,k_{i}}^{\transpose}=\bar\Delta_{i,3}, \\
	&\Delta_{i,4}\!=\!\frac{1}{N} X_{k_{i}}U_{k_{i}}^{\transpose} \xrightarrow{\text{a.s.}}  \frac{1}{N} \bar X_{k_{i}} \bar U_{k_{i}}^{\transpose}=\bar\Delta_{i,4},
\end{split}
\end{align}
where $\xrightarrow{\text{a.s.}}$ denotes almost sure convergence. Each matrix $Q_{i,c}$ can be written as a function of $\Delta_{i,j}$, $j=1,2,3$,~namely,
\begin{align*}
Q_{i,c}  = \ & (\Delta_{i,3}-\Delta_{i,3}(\Delta_{i,2}-\sigma_{U_{k_{i}}}^{2}I)^{\dag} \Delta_{i,3})\cdot\\
		&\cdot (\Delta_{i,1}-\sigma_{X_{0,k_{i}}}^{2}I+\Delta_{i,1}(\Delta_{i,2}-\sigma_{U_{k_{i}}}^{2}I)^{\dag} \Delta_{i,1})^{\dag}.
\end{align*}
Further, notice that $Q_{i,c}$ is continuous at  $\Delta_{i,j}=\bar\Delta_{i,j}$, $j=1,2,3$, since $[\bar X_{0,k_{i}}^{\transpose} \ \bar U_{k_{i}}^{\transpose}]^{\transpose}$ is full row rank by assumption. Thus, by \eqref{eq:LLN} and the Continuous Mapping Theorem \cite[Theorem 2.3]{AWV00}, as $N\to \infty$,
\begin{align}\label{eq:Qias}
	Q_{i,c}  \xrightarrow{\text{a.s.}} \bar{Q}_{i},
\end{align}
where $\bar Q_i = \bar X_{k_{i}}K_{\bar U_{k_{i}}}(\bar X_{0,k_{i}}K_{\bar U_{k_{i}}})^{\dag}$. 
Analogously, each $L_{i,c}$ can be written as
\begin{align*}
L_{i,c}  = \ & (\Delta_{i,4}-\Delta_{i,4}(\Delta_{i,1}-\sigma_{X_{0,k_{i}}}^{2}I)^{\dag} \Delta_{i,4})\cdot\\
		&\cdot (\Delta_{i,2}-\sigma_{U_{k_{i}}}^{2}I+\Delta_{i,2}(\Delta_{i,1}-\sigma_{X_{0,k_{i}}}^{2}I)^{\dag} \Delta_{i,2})^{\dag},
\end{align*}
and the same argument as before shows that, as $N\to \infty$,
\begin{align}\label{eq:Lias}
	L_{i,c}  \xrightarrow{\text{a.s.}} \bar{L}_{i},
\end{align}
where $\bar L_i = \bar X_{k_{i}}K_{\bar X_{0,k_{i}}}(\bar U_{k_{i}}K_{\bar X_{0,k_{i}}})^{\dag}$. 
Finally, by applying \eqref{eq:Qias}, \eqref{eq:Lias}, and, once again, the Continuous Mapping Theorem, we conclude that $u_{T,c}''\xrightarrow{\text{a.s.}} u_{T}^{*}$ as $N\to \infty$.
\end{proof}

Consider now the data-driven control \eqref{eq: u data driven}. After some algebraic manipulations, it can be rewritten~as
\begin{align}\label{eq: u data driven expansion}
	u^{*}_{T} & =\! (I - G \Pi_{\bar H} G^{\transpose}(G \Pi_{\bar H} G^{\transpose})^{\dag}) G\bar H^{\transpose} ({\bar H}\bar H^{\transpose})^{\dag}\! \! \begin{bmatrix} x_{0} \\ \subscr{x}{f}  \end{bmatrix}\!.
\end{align}
We introduce the following ``corrected'' version of \eqref{eq: u data driven expansion}:
\begin{align}\label{eq: u data driven corr}
	u_{T,c}' &\! =\! (I \!-\! (G_{c} \Pi_{\bar H,c} G_{c}^{\transpose}\!-\!N\sigma_{U}^{2}I)(G_{c} \Pi_{\bar H,c} G_{c}^{\transpose}\!-\!N\sigma_{U}^{2}I)_{\varepsilon}^{\dag}) \cdot \notag\\ 
	&\hspace{2.575cm} \cdot G_{c}\bar H_{c}^{\transpose} ({\bar H_{c}}\bar H_{c}^{\transpose}-\Delta_{\bar H})^{\dag}\! \! \begin{bmatrix} x_{0} \\ \subscr{x}{f}  \end{bmatrix}\!,
\end{align}
where $G_{c}$ and $\bar H_{c}$ are defined as $G$ and $\bar H$, after replacing all instances of $Q_i$ and $K_{X_{0,k_{i}}}$ with $Q_{i,c}$ and $\Pi_{X_{0,k_i},c}$, respectively, the operation $(\cdot)^{\dag}_{\varepsilon}$ is defined in Remark \ref{rem:numerics}, and
\begin{align*}
	\Pi_{\bar H,c} &= I-\bar H_{c}^{\transpose}(\bar H _{c}\bar H_{c}^{\transpose}-\Delta_{\bar H})^{\dag}\bar H_{c},\ \Delta_{\bar H} = \begin{bmatrix} 0_{n} & 0_{n} \\ 0_{n} & \Delta_{\bar H,2} \end{bmatrix},\\
	\Delta_{\bar H,2} & = N\sigma_{X}^{2} \sum_{j=0}^{\ell} \prod_{i=0}^{j-1}Q_{\ell-i,c}\left(\prod_{i=0}^{j-1}Q_{\ell-i,c}\right)^{\transpose}, \ Q_{0,c}=I.
\end{align*}
\begin{theorem}{\emph{\bfseries(Consistency of $u_{T,c}'$)}}\label{thm: consistency 2} Assume that $\mc D$ is corrupted by noise as in \eqref{eq:corrupted data}, and that $[\bar X_{0,k_{i}}^{\transpose} \ \bar U_{k_{i}}^{\transpose}]^{\transpose}$ is full row rank for all $i\in\until{\ell}$. For $\varepsilon>0$ sufficiently small, the data-driven control $u_{T,c}'$ in \eqref{eq: u data driven corr} converges almost surely to the minimum-energy control input $u^*_T$ as $N\to\infty$.
\end{theorem}

\begin{figure}
  \vspace{0.15cm}
  \includegraphics[width=.485\textwidth]{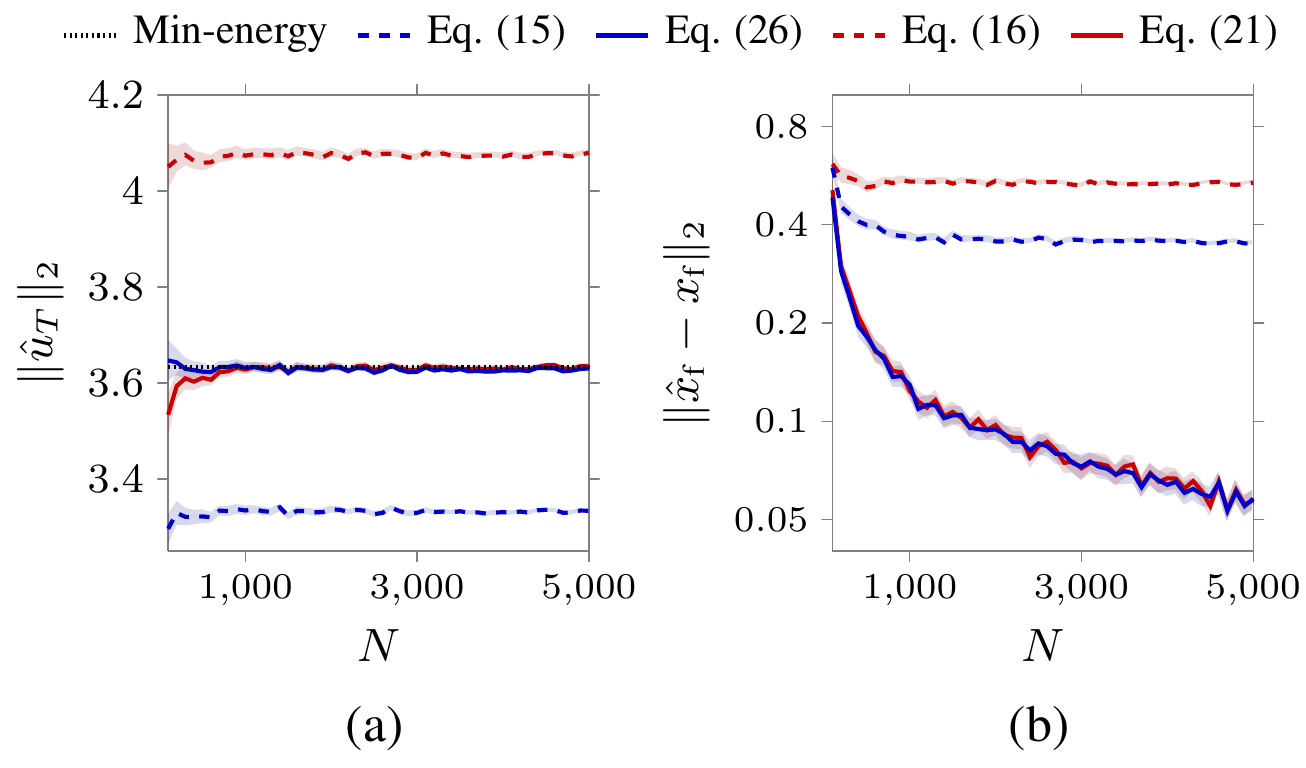}
  \vspace{-0.5cm}
\caption{In this figure we compare the behavior of the data-driven minimum-energy inputs \eqref{eq: u data driven} and \eqref{eq: alternative u data driven} and their corrected versions \eqref{eq: u data driven corr} and \eqref{eq: alt u data driven corr}, respectively. We choose a system of dimension $n=4$ with $m=2$ inputs. The system matrices $A$ and $B$ have been generated randomly with i.i.d.~normal entries.  Data have been divided into  $M=2$ datasets with time horizons $T_{1}=3$ and $T_{2}=4$ and $N_{1}=N_{2}=N$ measurements. We choose a control horizon $T=T_{1}+T_{2}=7$.  The entries of every $X_{0,i}$ and $U_{i}$ are independently and uniformly distributed in $[0,1]$. The entries of the initial state $x_{0}$ and final state $\subscr{x}{f}$ have been independently drawn from of a normal distribution. The plots show the norm of the minimum-energy input (panel (a)) and the corresponding error in the final state (panel (b)) for all the data-driven expressions as a function of the number of data $N$. For the data-driven inputs \eqref{eq: u data driven} (cf.~Remark \ref{rem:numerics}) and \eqref{eq: u data driven corr} we choose a tolerance $\varepsilon=10^{-8}$. The entries of all data matrices are corrupted by i.i.d.~Gaussian noise as in  \eqref{eq:corrupted data} with variance $\sigma_{X}^{2}=\sigma_{X_{0}}^{2}=\sigma_{U}^{2}=0.1$. {\color{black}The solid and dashed curves represent the average over 100 realizations of the noise, whereas the light-colored regions denote the 95\%~confidence~intervals~around~the~mean.} }
  \label{fig:numerical analysis noise}
\end{figure}

{\color{black} The proof of Theorem~\ref{thm: consistency 2} follows closely the one of Theorem \ref{thm: consistency} and is therefore omitted.} In Fig.~\ref{fig:numerical analysis noise}, we
illustrate the behavior of the data-driven expressions \eqref{eq: u
  data driven} and \eqref{eq: alternative u data driven}, and their
corrected versions \eqref{eq: u data driven corr} and \eqref{eq: alt u
  data driven corr}, respectively, as a function of the data size
$N$. 
Each dataset is corrupted
by i.i.d.~Gaussian noise as in \eqref{eq:corrupted data} with
$\sigma_{X}^{2}=\sigma_{X_{0}}^{2}=\sigma_{U}^{2}=0.01$. 
 As the number of data $N$ increases, the corrected data-driven
expressions \eqref{eq: u data driven corr} and \eqref{eq: alt u data
  driven corr} approach the minimum-energy cost
(Fig.~\ref{fig:numerical analysis noise}(a)) and the corresponding
errors in the final state decrease (Fig.~\ref{fig:numerical analysis
  noise}(b)), as predicted by Theorems \ref{thm: consistency} and
\ref{thm: consistency 2}.

\section{Conclusion}\label{sec:conc}
In this paper we address the problem of computing minimum-energy
controls for linear systems using heterogeneous data. Specifically, we
consider data consisting of input-state trajectories featuring
different time horizons and initial conditions. We derive two
different data-driven expressions of minimum-energy controls for a
wide range of control horizons, possibly different from those in the
experiments. When data are affected by i.i.d.~noise with zero mean and
known variance, we modify our expressions so to ensure
convergence to the correct controls in the limit of infinite~data. 

Directions for future work include the application of our
approach and data collection setting to other control problems, such as LQR and MPC, {\color{black} the sensitivity analysis of the corrected data-driven control inputs to
uncertainty in the noise variances, and the derivation of non-asymptotic bounds on the reconstruction error in the case of finite noisy data.}

\bibliographystyle{unsrt}
\bibliography{alias,FP,Main,New}

\end{document}